\documentclass[letterpaper,11pt,oneside,reqno]{amsart}

\usepackage[colorlinks=true,linkcolor=blue,citecolor=red]{hyperref}
\usepackage[alphabetic,nobysame]{amsrefs}

\usepackage{amsmath,amssymb,amsthm,amsfonts,mathtools}
\usepackage{graphicx,color}
\usepackage{upgreek}
\usepackage{esint}
\usepackage[mathscr]{euscript}
\allowdisplaybreaks
\numberwithin{equation}{section}

\usepackage{tikz}
\usetikzlibrary{shapes,arrows,positioning,decorations.markings}

\usepackage{ytableau}
\usepackage{array}
\usepackage{adjustbox}
\usepackage{enumerate}
\usepackage[margin=10pt]{caption}

\usepackage[DIV=12]{typearea}

\renewcommand{\le}{\leqslant}
\renewcommand{\ge}{\geqslant}

\usepackage{aliascnt}

\newaliascnt{lemma}{proposition}
\newtheorem{lemma}[lemma]{Lemma}
\aliascntresetthe{lemma}
\newaliascnt{corollary}{proposition}
\newtheorem{corollary}[corollary]{Corollary}
\aliascntresetthe{corollary}
\newaliascnt{theorem}{proposition}
\newtheorem{theorem}[theorem]{Theorem}
\aliascntresetthe{theorem}
\newaliascnt{conjecture}{proposition}

\aliascntresetthe{conjecture}
\newaliascnt{claim}{proposition}

\aliascntresetthe{claim}
\theoremstyle{definition}
\newaliascnt{definition}{proposition}
\newtheorem{definition}[definition]{Definition}
\aliascntresetthe{definition}
\newaliascnt{remark}{proposition}
\newtheorem{remark}[remark]{Remark}
\aliascntresetthe{remark}
\newaliascnt{example}{proposition}

\aliascntresetthe{example}
\usepackage{cleveref}
\crefname{proposition}{proposition}{propositions}
\Crefname{proposition}{Proposition}{Propositions}
\crefname{lemma}{lemma}{lemmas}
\Crefname{lemma}{Lemma}{Lemmas}
\crefname{corollary}{corollary}{corollaries}
\Crefname{corollary}{Corollary}{Corollaries}
\crefname{theorem}{theorem}{theorems}
\Crefname{theorem}{Theorem}{Theorems}
\crefname{conjecture}{conjecture}{conjectures}
\Crefname{conjecture}{Conjecture}{Conjectures}
\crefname{claim}{claim}{claims}
\Crefname{claim}{Claim}{Claims}
\crefname{definition}{definition}{definitions}
\Crefname{definition}{Definition}{Definitions}
\crefname{remark}{remark}{remarks}
\Crefname{remark}{Remark}{Remarks}
\crefname{example}{example}{examples}
\Crefname{example}{Example}{Examples}

\begin{document}

\title{Random permutations from $q$-Demazure products}

\author{Mikhail Tikhonov}

\date{}

\begin{abstract}
 We study the $q$-deformation of the Demazure product model from \cite{GrothendieckShenanigans2024}. Consider the longest element $w_0$ in $S_n$ decomposed into simple transpositions.

\begin{equation*}
    w_0 = s_{n-1} (s_{n-2} s_{n-1}) \cdots (s_1 s_2 \cdots s_{n-1})
\end{equation*}
Then, independently, delete each transposition with probability $1-p$ (equivalently, keep it with probability $p$) and apply the $q$-Demazure product to the remaining ones. We show that the law of the resulting permutation converges as $n \to \infty$ to a \emph{deterministic limiting measure}, known as a \emph{permuton}. Moreover, this limiting permuton coincides with the special case ($q=0$) studied in \cite{GrothendieckShenanigans2024} for adjusted probability $p'=\frac{p(1-q)}{1-qp}$. This resolves \cite{GrothendieckShenanigans2024}*{Conjecture~1.13} and identifies the limiting permuton explicitly.
\end{abstract}

\maketitle

\section{Introduction}

\subsection{Overview}

The symmetric group $S_n$ is generated by simple transpositions
$s_1, \dots, s_{n-1}$, where $s_i$ swaps positions $i$ and $i+1$.
We consider the following deformation of the group product.

\begin{definition}[$q$-Demazure product]
\label{def:q_demazure}
Fix $q \in [0,1]$.
The \emph{$q$-Demazure product} of a simple transposition $s_i$ and a permutation
$\sigma \in S_n$ is the random permutation
\[
  s_i \cdot_q \sigma \;:=\;
  \begin{cases}
    s_i\sigma, & \text{if } \ell(s_i\sigma) > \ell(\sigma),\\[6pt]
    \begin{cases}
      s_i\sigma & \text{with probability } q,\\
      \sigma    & \text{with probability } 1-q,
    \end{cases}
    & \text{if } \ell(s_i\sigma) < \ell(\sigma),
  \end{cases}
\]
where $\ell(\sigma) = \#\{(a,b) : a < b,\, \sigma(a) > \sigma(b)\}$ is the number of
inversions of $\sigma$.
The $q$-Demazure product of a word $s_{i_1}\cdot_q \cdots \cdot_q s_{i_k}$ is defined by applying the rule
left-to-right with independent randomness at each step.
At $q=0$ this reduces to the \emph{Demazure product}; at $q=1$ it is the ordinary
group multiplication in $S_n$.
\end{definition}

Consider the longest element $w_0 \in S_n$ and
its standard reduced decomposition into simple transpositions:
\begin{equation*}
	w_0 = s_{n-1} (s_{n-2} s_{n-1}) \cdots (s_1 s_2 \cdots s_{n-1}).
\end{equation*}
We independently delete each letter with probability $1 - p$
(equivalently, keep it with probability~$p$) and apply the
$q$-Demazure product to the remaining subword, obtaining a
random permutation $\sigma_n \in S_n$.
We address the question: what is the large-$n$ behavior of~$\sigma_n$?

The two boundary cases $q=0$ and $q=1$ have been studied previously.
When $q = 0$ (the ordinary Demazure product),
$\sigma_n$ converges to a deterministic permuton $\mu_p$
\cite{GrothendieckShenanigans2024};
subsequently, Defant \cite{defant2025permutons} extended this theory to more general classes of reduced decompositions.
When $q = 1$ (ordinary group multiplication),
$\sigma_n$ is the product of a random subword of $w_0$;
the expected number of inversions in this regime was computed in \cite{Defant2024}.
Thus \cref{thm:intro_main} addresses the intermediate regime $q \in (0,1)$,
whose convergence to a permuton was conjectured in
\cite{GrothendieckShenanigans2024}*{Conjecture~1.13}.
Our result complements
\cite{defant2025permutons}: Defant varies the reduced decomposition at $q=0$, whereas we fix the staircase decomposition of $w_0$
and vary~$q$.

\begin{remark}[Terminology: Hecke vs.\ $q$-Demazure]
  \label{rem:terminology}
  The combined keep-or-delete-then-Demazure step at each position
  is precisely the left multiplication $T_\sigma \mapsto R_k(p) \cdot T_\sigma$
  in the Hecke algebra $H_n(q)$
  with relation $(T_i + q)(T_i - 1) = 0$,
  where $R_k(p) := p\,T_{s_k} + (1-p)$
  \cite{galashin2020symmetries}*{Proposition~2.3}.
  In particular, the parameter $q$ enters the distribution not only as a structural deformation of the algebra but also as a probabilistic weight: at a reducing step, the crossing probability is $qp$ rather than~$p$.
\end{remark}

The product of $R_k(p)$ operators over the standard reduced word of~$w_0$ is the homogeneous specialization $x_1 = \cdots = x_N = p$ of the \emph{Yang--Baxter basis} element $Y^{w_0}(x_1,\dots,x_N)$ introduced by Lascoux, Leclerc, and Thibon \cite{lascoux1997flag}.
Galashin \cite{galashin2020symmetries} showed that the coefficient of~$T_\pi$ in $Y^w(x_1,\dots,x_N)$ equals the probability of observing color permutation~$\pi$ in the stochastic colored six-vertex model on the wiring diagram of~$w$ with inhomogeneous rapidities $x_1,\dots,x_N$.
This identification is closely related to the color-position symmetry of Borodin--Bufetov \cite{BorodinBufetov2021ColorPosition} and the interacting-particle perspective of Bufetov \cite{bufetov2020interacting}.

The stochastic colored six-vertex model underlying our construction has rich algebraic symmetries.
Borodin, Gorin, and Wheeler \cite{borodin2019shift} discovered \emph{shift-invariance}: certain distributional identities for height functions that imply analogous statements for the KPZ equation, directed polymers, and last-passage percolation.
Galashin \cite{galashin2020symmetries} identified a more fundamental symmetry --- \emph{flip-invariance} --- and showed that shift-invariance is a special case (a composition of two flips); these identities hold for inhomogeneous rapidities and arbitrary skew domains.
The colored vertex weights we define in \cref{sec:model} are a homogeneous specialization of those appearing in these works, restricted to a triangular domain with rainbow boundary.

Our proof strategy, however, does not use these symmetries; see \cref{subsec:methods}.

\subsection{Results}

We state the main result in terms of the height function of a permutation.

\begin{definition}[Height function of a permutation]
\label{def:intro_h_perm}
For $\sigma \in S_n$ and $(x, y) \in [0,1]^2$, the \emph{height function} of $\sigma$ is
\[
  H_\sigma(x,y) \;:=\; \frac{1}{n}\,\#\bigl\{\, i \le \lfloor nx \rfloor : \sigma(i) > \lfloor ny \rfloor \,\bigr\}.
\]
Equivalently, $H_\sigma(x,y)$ is the normalized number of points of $\sigma$
(viewed as a subset of $\{1,\dots,n\}^2$) in the upper-left rectangle $[1, \lfloor nx\rfloor] \times (\lfloor ny\rfloor, n]$.
\end{definition}

\begin{theorem}[Main result; see \Cref{thm:main}]
\label{thm:intro_main}
Fix $q \in (0,1)$ and $p \in (0,1)$.
Let $\sigma_n \in S_n$ be the random permutation obtained by the $q$-Demazure product
with deletion probability $1-p$ applied to the standard reduced decomposition of~$w_0$.
Then:
\begin{enumerate}[\upshape(i)]
  \item The sequence of permutons $\mu_{\sigma_n}$ converges weakly, in probability,
  to a deterministic permuton $\mu_{p,q}$.
  \item $\mu_{p,q}$ coincides with the limiting permuton $\mu_{p'}$ from
  \cite{GrothendieckShenanigans2024} for the Demazure product ($q=0$) with adjusted parameter
  \[
    p' \;=\; \frac{p(1-q)}{1-qp}.
  \]
  \item The explicit height function of $\mu_{p,q}$ is given in \cref{eq:h_refl,eq:height_rescaling}.
\end{enumerate}
\end{theorem}

The adjusted parameter $p'$ arises from matching the six-vertex model parameter
$\kappa = (1-qp)/(1-p)$ between the two models; see the proof of \cref{thm:main}.
Note that $p' \to p$ as $q \to 0$, recovering the Demazure case,
and at $p = 1$ the input word is reduced, so no reducing steps arise and $\sigma_n = w_0$ for all~$q$.

\subsection{Strategy of proof}
\label{subsec:methods}

The key observation is that the hydrodynamic limit of the stochastic six-vertex model depends on the vertex weights only through the parameter $\kappa = (1-qp)/(1-p)$, which coincides with the parameter of the $q=0$ model at $p' = p(1-q)/(1-qp)$.
The challenge is that the model lives on a triangular domain with a reflecting boundary, for which no direct hydrodynamic limit theorem is available.
We overcome this via a monotone sandwich argument.

In \cref{sec:model}, we encode the $q$-Demazure random permutation as a \emph{colored stochastic
six-vertex model} (S6V) on a triangular domain $\delta \subset \mathbb{Z}^2$
with rainbow boundary conditions, adapting the framework of
\cite{GrothendieckShenanigans2024}*{Section~2} to $q > 0$.
The colored vertex weights $L_{p,q}$ (\cref{eq:colored6Vweights}) extend the $q = 0$ weights of \cite{GrothendieckShenanigans2024} by allowing length-decreasing crossings with probability~$qp$.
By the color-forgetting property, for each color threshold $X$, the
colored height function reduces to that of an \emph{uncolored} S6V model with
weights $w_{p,q}$ (\cref{eq:uncolored6Vweights}), evolving on $\delta$
with a reflecting boundary.

In \cref{sec:hydro}, we analyze the stochastic six-vertex model on the discrete cylinder
$\mathfrak{C}_{M,L}$ with double-step initial data.
By the hydrodynamic limit theorem of Aggarwal \cite{aggarwal2020limit},
the particle density converges to the entropy solution of a conservation
law with rational flux $\varphi(z) = \kappa z / ((\kappa - 1)z + 1)$.
We construct this entropy solution
explicitly: it consists of the BCG rarefaction fan \cite{BCG6V} joined
to a shock wave propagating along a computable curve.

In \cref{sec:main}, we construct a monotone coupling
to sandwich the active height function of the
reflective system between the S6V on the cylinder with double-step initial
data (lower bound) and the step-initial S6V on the quadrant (upper bound).
Both bounds converge to the same deterministic limit, which we identify
with the permuton height function from \cite{GrothendieckShenanigans2024}
under the parameter substitution $p \mapsto p'$.

\subsection{Acknowledgments}
We are grateful to Amol Aggarwal, Vadim Gorin, Greta Panova, and Leonid Petrov for helpful discussions. 
The work was partially supported by NSF grant DMS-2153869.

\section{From permutations to the colored stochastic six-vertex model}

\label{sec:model}

\subsection{From permutations to vertex models}

Let $\delta = \{(S,Y) : 1 \le Y < S \le n\}$ be the triangular region on the square grid.
We introduce the colored vertex weights $L_{p,q}$ as follows. These encode the $q$-Demazure product rule: when $i > j$ (non-reducing step), the straight configuration has probability $p$; when $j > i$ (reducing step), it has probability $qp$.

Since only the relative order of the two incoming colors matters, every non-trivial configuration can be depicted using two representative colors. In the pictures below we use {\color{red}red} for the larger incoming color and black for the smaller.

\begin{equation}
\label{eq:colored6Vweights}
  \begin{gathered}
L_{p,q}\!\left(\,
\begin{tikzpicture}[baseline=-3,scale=.7,very thick]
    \draw[fill] (0,0) circle [radius=0.05];
    \draw[black] (0,-0.05) -- (0,-0.5); \draw[black] (0,0.05) -- (0,0.5);
    \draw[red]   (-0.05,0) -- (-0.5,0); \draw[red]   (0.05,0) -- (0.5,0);
\end{tikzpicture}
\,\right) = qp,
\qquad
L_{p,q}\!\left(\,
\begin{tikzpicture}[baseline=-3,scale=.7,very thick]
    \draw[fill] (0,0) circle [radius=0.05];
    \draw[black] (0,-0.05) -- (0,-0.5); \draw[black] (0.05,0) -- (0.5,0);
    \draw[red]   (-0.05,0) -- (-0.5,0); \draw[red]   (0,0.05) -- (0,0.5);
\end{tikzpicture}
\,\right) = 1-qp,
\qquad (j > i), \\[10pt]
L_{p,q}\!\left(\,
\begin{tikzpicture}[baseline=-3,scale=.7,very thick]
    \draw[fill] (0,0) circle [radius=0.05];
    \draw[red]   (0,-0.05) -- (0,-0.5); \draw[red]   (0,0.05) -- (0,0.5);
    \draw[black] (-0.05,0) -- (-0.5,0); \draw[black] (0.05,0) -- (0.5,0);
\end{tikzpicture}
\,\right) = p,
\qquad
L_{p,q}\!\left(\,
\begin{tikzpicture}[baseline=-3,scale=.7,very thick]
    \draw[fill] (0,0) circle [radius=0.05];
    \draw[red]   (0,-0.05) -- (0,-0.5); \draw[red]   (0.05,0) -- (0.5,0);
    \draw[black] (-0.05,0) -- (-0.5,0); \draw[black] (0,0.05) -- (0,0.5);
\end{tikzpicture}
\,\right) = 1-p,
\qquad (i > j).
\end{gathered}
\end{equation}
When $i = j$, only the straight configuration is possible and $L_{p,q}(i,i;i,i) = 1$.

Place a stochastic vertex with the colored weight $L_{p,q}$ from \eqref{eq:colored6Vweights} at each vertex $(S,Y) \in \delta$.
Along the diagonal $Y = S$, the domain $\delta$ has no vertices: the rightmost vertex in row~$S$ is at $Y = S - 1$. An arrow exiting horizontally from $(S, S-1)$ reaches the diagonal, where it is forced to turn upward and becomes the bottom input of $(S+1, S)$. These diagonal nodes act as a \emph{reflecting boundary} for the model.
We equip the model with \emph{rainbow} boundary conditions: color $r$ enters from the left at row $S = n + 1 - r$ for $r = 1, \dots, n$ (see \Cref{fig:rainbow_initial}).

\begin{figure}
\centering
\includegraphics[width=0.85\textwidth]{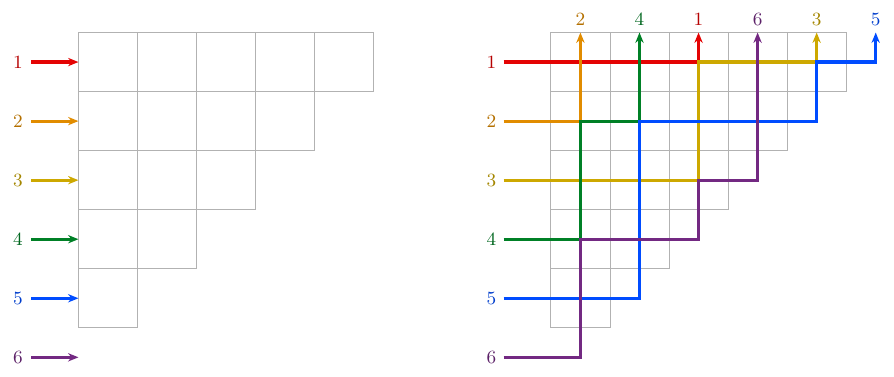}
  \caption{The colored stochastic six-vertex model on the triangular domain $\delta$ with $n=6$.
  \emph{Left:} Rainbow initial conditions — arrows of colors $1,2,\dots,6$ enter from the left boundary.
  \emph{Right:} A sample configuration; each vertex independently applies the weight $L_{p,q}$ from \eqref{eq:colored6Vweights}.
  The permutation is read from the colors of arrows at the top row $S = n$.}
  \label{fig:rainbow_initial}
\end{figure}

A random configuration is sampled row-by-row, from $S = 2$ to $S = n$, processing each row left-to-right (increasing $Y$). At each vertex, the bottom and left inputs are already determined, so the output can be sampled independently according to $L_{p,q}$. Reading the colors of vertical arrows at the top row $S = n$, we arrive at a random permutation.

\subsection{Pipe dream interpretation}
\label{subsec:pipedream}

The random configuration is equivalently a \emph{pipe dream} (or \emph{RC-graph}), a standard object in Schubert calculus. A pipe dream is a tiling of $\delta$ where each box $(S,Y)$ carries one of two tiles:
\begin{itemize}
  \item a \emph{cross}
    \raisebox{-3pt}{\begin{tikzpicture}[scale=0.45, line width=0.8pt]
      \draw[gray!40] (0,0) rectangle (1,1);
      \draw (0.5,0) -- (0.5,1);
      \draw (0,0.5) -- (1,0.5);
    \end{tikzpicture}}
    — the two pipes pass straight through each other, corresponding to the straight configuration of $L_{p,q}$;
  \item a \emph{bump}
    \raisebox{-3pt}{\begin{tikzpicture}[scale=0.45, line width=0.8pt]
      \draw[gray!40] (0,0) rectangle (1,1);
      \draw (0.5,0) to[out=90,in=180] (1,0.5);
      \draw (0,0.5) to[out=0,in=-90] (0.5,1);
    \end{tikzpicture}}
    — the two pipes turn and avoid crossing, corresponding to the turned configuration.
\end{itemize}
Under the rainbow initial conditions, pipe $r$ enters from the left at row $S = n+1-r$. Each pipe follows the tiling: it passes straight through every cross it meets and turns at every bump. The color exiting at the top row $S = n$ at column $Y$ is $\sigma(Y)$.

Stripping away colors, a pipe dream is simply a $\{$cross, bump$\}$-filling of $\delta$. The crosses, read row-by-row from $S=2$ to $S=n$ (left-to-right within each row), spell out a word $w = s_{m_1} s_{m_2} \cdots s_{m_k}$ in the simple transpositions of $S_n$, where the transposition at box $(S,Y)$ is $s_{n+Y-S}$. The permutation $\sigma$ encoded by the tiling is recovered by the \emph{Demazure product}. Starting from $\sigma = \mathrm{id}$, apply each $s_{m_t}$ in order. If $\ell(s_{m_t} \sigma) > \ell(\sigma)$ (a non-reducing step), update $\sigma \leftarrow s_{m_t} \sigma$; otherwise skip. This gives $\sigma = \partial_w(\mathrm{id})$.

The $q$-reduction \cite{GrothendieckShenanigans2024} replaces the deterministic skip by a coin flip: at a reducing cross ($\ell(s_{m_t}\sigma) < \ell(\sigma)$), apply $s_{m_t}$ with probability $q$ and skip with probability $1-q$. This recovers the $q$-Demazure product rule of \cref{def:q_demazure}. In the tiling picture, a reducing cross at $(S,Y)$ occurs when the two pipes meeting there have already crossed elsewhere in the diagram. Each such excess crossing contributes an independent factor of~$q$ to the weight of the configuration.

\begin{definition}[Colored height function]
\label{def:colored_height}
For a color threshold $c \in \{0, 1, \dots, n\}$, a row $S$, and a column $Y$, define
\[
  H(c;\, S,\, Y)
  \;:=\;
  \#\bigl\{\, Y' \ge Y : \text{the vertical arrow at column } Y',\, \text{row } S
  \text{ has color } \le c \,\bigr\}.
\]
At the top row $S = n$, this is related to the permutation height function by
$H(c;\, n,\, Y) = c - Y + 1 + n \cdot H_\sigma\!\bigl(\tfrac{Y-1}{n},\, \tfrac{c}{n}\bigr),$
since $H(c;\, n,\, Y)$ counts colors $\le c$ at positions $Y' \ge Y$, while $H_\sigma$ counts positions $i \le \lfloor nx \rfloor$ with $\sigma(i) > \lfloor ny \rfloor$.
\end{definition}

Fix an integer parameter $1 \leq X \leq n$. We erase all pipes of colors $>X$ and forget the color codes of all other arrows.
Since the colored weights \eqref{eq:colored6Vweights} depend only on the relative order of the two incoming colors, the resulting model evolves as an uncolored stochastic six-vertex model with the following weights.
Here $i, j, k, l \in \{0,1\}$ are the \emph{edge occupation numbers} of the four
edges of a vertex: $i$ (bottom, incoming), $j$ (left, incoming),
$k$ (top, outgoing), $l$ (right, outgoing), with $1$ indicating an arrow is present.
The weights are nonzero only when $i + j = k + l$ (arrow conservation):

\begin{equation*}
w_{p,q}(i,j;k,l)
=
w_{p,q}
\bigl(
	\begin{tikzpicture}[baseline=-3,scale=.7,very thick]
			\draw[fill] (0,0) circle [radius=0.025];
			\draw (0.5,0) -- (0.05,0);
			\draw (-0.5,0) -- (-0.05,0);
			\draw (0,0.05) -- (0, 0.5);
			\draw (0,-0.05) -- (0,-0.5);
			\node at (.25,-.4) {\scriptsize $i$};
			\node at (-.7,0) {\scriptsize $j$};
			\node at (-.25,.4) {\scriptsize $k$};
			\node at (.7,0) {\scriptsize $l$};
\end{tikzpicture}
\bigr),
\end{equation*}

\begin{equation}
w_{p,q}
\bigl(
	\begin{tikzpicture}[baseline=-3,scale=.7,very thick]
			\draw[fill] (0,0) circle [radius=0.05];
			\draw[thick] (0.5,0) -- (0.05,0);
			\draw[thick] (-0.5,0) -- (-0.05,0);
			\draw[dashed] (0,0.05) -- (0, 0.5);
			\draw[dashed] (0,-0.05) -- (0,-0.5);
\end{tikzpicture}
\bigr)
= p, \quad w_{p,q}
\bigl(
	\begin{tikzpicture}[baseline=-3,scale=.7,very thick]
			\draw[fill] (0,0) circle [radius=0.05];
			\draw[thick] (0,0.05) -- (0, 0.5);
			\draw[thick] (0,-0.05) -- (0,-0.5);
			\draw[dashed] (0.5,0) -- (0.05,0);
			\draw[dashed] (-0.5,0) -- (-0.05,0);
\end{tikzpicture}
\bigr) = qp, \quad w_{p,q}
\bigl(
	\begin{tikzpicture}[baseline=-3,scale=.7,very thick]
			\draw[fill] (0,0) circle [radius=0.05];
			\draw[dashed] (0,0.05) -- (0, 0.5);
			\draw[thick] (0,-0.05) -- (0,-0.5);
			\draw[thick] (0.5,0) -- (0.05,0);
			\draw[dashed] (-0.5,0) -- (-0.05,0);
	\end{tikzpicture}
\bigr) = 1 - qp, \quad w_{p,q}
\bigl(
	\begin{tikzpicture}[baseline=-3,scale=.7,very thick]
			\draw[fill] (0,0) circle [radius=0.05];
			\draw[thick] (0,0.05) -- (0, 0.5);
			\draw[dashed] (0,-0.05) -- (0,-0.5);
			\draw[dashed] (0.5,0) -- (0.05,0);
			\draw[thick] (-0.5,0) -- (-0.05,0);
	\end{tikzpicture}
\bigr) = 1 - p.
\label{eq:uncolored6Vweights}
\end{equation}

Consider the uncolored height function, defined as
\begin{equation}
\label{eq:uncolored_height}
    H_X(S,Y) \;:=\; \#\bigl\{\, Y' \ge Y : \text{there is a vertical arrow at column } Y',\, \text{row } S \,\bigr\}.
\end{equation}

Our goal is to understand the asymptotics of the uncolored height function $H_X(\lfloor{n s}\rfloor,\lfloor{n y}\rfloor)$,
where $(\lfloor{n s}\rfloor,\lfloor{n y}\rfloor)$ are the coordinates inside the region $\delta$.
In \cref{sec:coupling} we couple $H_X$ monotonically between two S6V models with known limit shapes; in \cref{sec:permuton} we identify the sandwich limit as the permuton from \cite{GrothendieckShenanigans2024} with adjusted parameter $p'$.

\section{Hydrodynamics on the cylinder}
\label{sec:hydro}

We set up the stochastic six-vertex model on the cylinder and compute the limit shape for double-step initial data, using the hydrodynamic limit theorem of \cite{aggarwal2020limit}.

\subsection{Model Setup}
Let $M$ be a positive integer and set $L = \lfloor \lambda M \rfloor$ for a fixed parameter $\lambda \in (0, \frac12]$. In the application to our model, $\lambda = x/(1+x)$ for $x \in (0,1)$; see \cref{eq:lambda_def}.
Define the discrete circle (one-torus) $\mathfrak{T}_M = \mathbb{Z} / M\mathbb{Z}$
and the discrete cylinder $\mathfrak{C}_{M,L} = \mathfrak{T}_M \times \{1, 2, \dots, L\}.$
We view the cylinder $\mathfrak{C}_{M,L}$ as a graph, where the vertices $(u_1, v_1), (u_2,v_2) \in \mathfrak{C}_{M,L}$ are connected if $(u_1-u_2, v_1-v_2) \in \{(-1,0), (1,0), (0,-1), (0,1)\}$.

With each vertex $\mathfrak{v} \in \mathfrak{C}_{M,L}$ we associate an arrow configuration $(i_1, j_1; i_2, j_2)$ and assign the stochastic six-vertex weight $w_{p,q}(i_1, j_1; i_2, j_2)$ according to \cref{eq:uncolored6Vweights}.
Arrow configurations are consistent: inputs at each vertex are outputs of its neighbors.

Finally, we equip the cylinder with initial data $\psi^{(M)}(i) \in \{0,1\}$. The upper end of the cylinder has \emph{free boundary conditions}.

We are interested in the \emph{thermodynamic limit} of this model, i.e. as $M \to \infty$.
Let $\nu_j(i)$ be an indicator for the event that there is a vertical arrow from $(i,j)$ to $(i,j+1)$.
In particular, the boundary configuration translates to $\nu_0(u) = \psi^{(M)} (u)$ for each $u \in \mathfrak{T}_M$.

Let \[\kappa  = \frac{1 - q p}{1 - p} > 1\] and
define the \emph{flux} function $\varphi: [0,1] \to [0,1]$ by
\begin{equation}
  \label{eq:flux_def}
  \varphi (z)  := \frac{\kappa z}{(\kappa - 1)z + 1}.
\end{equation}

\subsection{Step initial data on the quadrant}
The standard stochastic six-vertex model on the quadrant with step initial
data was studied in \cite{BCG6V}. Our formulas are obtained from those of \cite{BCG6V} by reflecting the spatial coordinate and complementing the height function. We denote the resulting rarefaction profile
by $G_{BCG}$; its explicit formula is recorded in \eqref{eq:gbcg_def} below.
In \cref{sec:coupling}, we use this profile as an upper bound for the active
particles of the reflective system.

\subsection{Double-step initial data}
We now introduce \emph{double-step initial data}.
Define the function $\Psi: \mathbb{T} \to [0,1]$ by
\begin{equation}
  \label{eq:initial_conditions_double_step}
  \Psi (u) = 1 - \mathbf{1}_{0 \leq u  \leq 1 - 2 \lambda}.
\end{equation}
For each $M \geq 1$, denote the corresponding boundary conditions on $\mathfrak{C}_{M,L}$ by $\psi^{(M)} (i)$.
We will assume they satisfy the following property:
\begin{equation}
  \lim_{M \to \infty} \sup_{0 \leq u_1 \leq u_2 \leq 1} \left| \frac{1}{M} \sum_{j=\lfloor u_1 M \rfloor}^{\lfloor u_2 M \rfloor} \psi^{(M)} (j) - \int_{u_1}^{u_2} \Psi(u) \,du \right| = 0.
\end{equation}

Under this assumption, the following theorem (conjectured by \cite{GwaSpohn1992}, proven by Aggarwal in \cite{aggarwal2020limit}) identifies the deterministic hydrodynamic profile $G(u,v)$ governing macroscopic space-time empirical densities on $\mathbb{T} \times [0, \lambda]$.

\begin{theorem}[\cite{aggarwal2020limit}, Theorem 1.1]
  Fix $\varepsilon>0$ and let $G(u,v)$ be the entropy solution to the conservation law
  \begin{equation}
    \label{eq:conservation_pde}
    \frac{\partial}{\partial v} G(u,v) + \frac{\partial}{\partial u} \left( \varphi(G(u,v))\right) = 0,
  \end{equation}
  on the cylinder $\mathbb{T} \times [0, \lambda]$, with initial data given by $G(u,0) = \Psi(u)$ and the flux $\varphi(\cdot)$ defined as in \cref{eq:flux_def}.
  For each $M \geq 1$, let $\nu^{(M)} = (\nu_v (u))$ be the particle configuration associated to the sampled stochastic six-vertex configuration on the cylinder $\mathfrak{C}_{M,L}$ described above.
  Then,
  \begin{equation}
    \lim_{M \to \infty} \mathbb{P} \left[ \max_{\substack{0 \leq U_1 \leq U_2 < M \\ 0 \leq V_1 \leq V_2 \leq L}} \left|\frac{1}{M^2}\sum_{j=V_1}^{V_2} \sum_{i=U_1}^{U_2} \nu_{j} (i) - \int_{V_1/M}^{V_2/M}\int_{U_1/M}^{U_2/M} G(u,v) \,du \,dv \right| > \varepsilon \right] = 0.
  \end{equation}

\end{theorem}

\subsection{Explicit solution}
For the double-step initial data, the entropy solution can be constructed explicitly.
Define the \emph{reflection point}
\begin{equation}\label{eq:v_rp}
  v_{RP} \;=\; \frac{1-2\lambda}{\kappa - 1},
\end{equation}
the time at which the leading edge of the rarefaction fan ($u = v\kappa$)
first reaches the boundary line $u = v + 1 - 2\lambda$.

Let $\mathcal{V}$ be the piecewise curve in $\mathbb{T} \times [0,\lambda]$ defined by:
\begin{equation}
  \label{eq:def_curve}
  \mathcal{V}(v) \;=\;
  \begin{cases}
    v + 1 - 2\lambda, & 0 \le v \le v_{RP}, \\[6pt]
    \displaystyle\left( \sqrt{\frac{v}{\kappa}}
      + \sqrt{\frac{(1-2\lambda)(\kappa-1)}{\kappa}} \right)^{\!2},
    & v_{RP} < v \le \lambda.
  \end{cases}
\end{equation}
For $v \le v_{RP}$ the curve follows the boundary;
for $v > v_{RP}$ it detaches into the interior.
When $v_{RP} > \lambda$ (i.e.\ $x < 1/\kappa$),
Phase~2 does not occur and $\mathcal{V}(v) = v + 1 - 2\lambda$ throughout.
The two pieces join continuously at $v = v_{RP}$, where both give
$u = \kappa(1-2\lambda)/(\kappa-1)$.

Define the density function $G_{BCG}(u,v)$ by:

\begin{equation}
\label{eq:gbcg_def}
  G_{BCG}(u,v) :=
  \begin{cases} 1, \qquad \frac{v}{u} > \kappa, \\[6pt]
  \frac{1}{\kappa -1} \frac{(\sqrt{v \kappa} - \sqrt{u})}{\sqrt{u}}, \qquad \kappa^{-1} \leq \frac{v}{u} \leq \kappa, \\[6pt]
  0, \qquad \frac{v}{u} < \kappa^{-1}.
  \end{cases}
\end{equation}

The corresponding height function $h_{BCG}(u,v) = \int_u^{v\kappa} G_{BCG}(u',v)\,du'$ counts arrows to the \emph{right} of position~$u$, so $G_{BCG} = -\partial_u h_{BCG} \ge 0$.
From \eqref{eq:gbcg_def}, one computes
\begin{equation}\label{eq:hbcg_explicit}
  h_{BCG}(u,v) \;=\;
  \begin{cases}
    v - u, & u \le v/\kappa, \\[4pt]
    \displaystyle\frac{(\sqrt{v\kappa} - \sqrt{u})^2}{\kappa - 1},
      & v/\kappa < u \le v\kappa, \\[6pt]
    0, & u > v\kappa.
  \end{cases}
\end{equation}
In particular, $h_{BCG}(0, v) = v$.

Finally define $G_{shock}(u,v): \mathbb{T} \times [0, \lambda] \to [0,1]$:
\begin{equation}
  G_{shock} (u,v) := \begin{cases}
    G_{BCG} (u,v), & 0 \leq u \leq \mathcal{V}(v), \\
    1, & \mathcal{V}(v) < u < 1.
  \end{cases}
\end{equation}

\begin{theorem}\label{thm:entropy_solution}
  $G_{shock} (u,v)$ is the entropy solution to the conservation law \cref{eq:conservation_pde} and the curve $\mathcal{V}$ from \cref{eq:def_curve} is the shock propagation trajectory.
\end{theorem}
\begin{proof}
The flux $\varphi$ is smooth on $[0,1]$, so Kru{\v{z}}kov's theorem \cite{kruzkov1970first} ensures existence and uniqueness of the entropy solution.

Both $G_1 (u,v) = 1$ and $G_{BCG}(u,v)$ satisfy the differential equation.
There are two points of discontinuity in the initial data, at $u=0$ and $u=1-2\lambda$.
Since the flux is increasing and strictly concave, discontinuities propagate to the right; for $u<0$ we have $G_1(u,v) = 1$.

At $u = 0$, the initial data jumps downward from $1$ to $0$.
For strictly concave flux, the entropy solution is a rarefaction fan, realized by $G_{BCG}$ (see \cite{BCG6V}).

At $u = 1-2\lambda$, the initial data jumps upward from $0$ to $1$.
For strictly concave flux, such jumps produce entropy shocks.
The shock trajectory $\mathcal{V}(v)$ has two phases:

\emph{Phase~1} ($0 \le v \le v_{RP}$).
While the rarefaction fan has not yet reached the shock,
the density to the left of the discontinuity remains~$0$.
The Rankine--Hugoniot condition gives
$\frac{d\mathcal{V}}{dv} = \frac{\varphi(1) - \varphi(0)}{1 - 0} = 1,$
so the shock moves at unit speed: $\mathcal{V}(v) = v + 1 - 2\lambda$.

\emph{Phase~2} ($v_{RP} < v \le \lambda$).
Once the rarefaction reaches the shock at $v = v_{RP}$,
the density on the left becomes $G_{BCG}(\mathcal{V}(v), v) > 0$.
The Rankine--Hugoniot condition becomes
\[
\frac{d\mathcal{V}}{dv} = \frac{1}{(\kappa-1)G_{BCG}(\mathcal{V}(v), v) + 1}
\]
with initial condition $\mathcal{V}(v_{RP}) = \kappa(1-2\lambda)/(\kappa-1)$.
In the rarefaction region, $(\kappa-1)G_{BCG}(u,v) + 1 = \sqrt{v\kappa/u}$,
so the ODE simplifies to $d\mathcal{V}/dv = \sqrt{\mathcal{V}/(v\kappa)}$,
which is separable.
Solving yields the quadratic piece in \eqref{eq:def_curve}.
The Lax entropy condition is satisfied: the shock separates $G_{BCG} < 1$ (left) from $1$ (right), and since $\varphi$ is concave, $\varphi'(G_{BCG}) > \dot{\mathcal{V}} > \varphi'(1)$.

Since both rarefaction and shock portions are entropy-admissible and the construction matches the initial data, $G_{\mathrm{shock}}$ is the desired entropy solution.
\end{proof}

\begin{corollary}\label{cor:arrow_conservation}
The shock curve $\mathcal{V}(v)$
from \eqref{eq:def_curve} satisfies
\begin{equation}\label{eq:arrow_conservation}
  h_{BCG}\!\bigl(\mathcal{V}(v),\, v\bigr)
  \;=\; v + 1 - 2\lambda - \mathcal{V}(v).
\end{equation}
\end{corollary}

\section{Couplings and limit shapes}
\label{sec:main}

We embed the color-forgotten model from \cref{sec:model}
into the cylinder and carry out the sandwich argument.

\subsection{Trapezoid coordinates and cylinder embedding}
\label{sec:torus}

Let $x \in (0,1)$ be such that $X = \lfloor xn \rfloor$.
After color-forgetting with threshold $X$,
the colored S6V on the triangle $\delta$ becomes an uncolored S6V with
weights $w_{p,q}$ from \eqref{eq:uncolored6Vweights}.

We focus on the \emph{trapezoid}: the subregion of $\delta$ where $S$
ranges from $n - X + 1$ to~$n$.
Introduce \emph{trapezoid coordinates}
\begin{equation}\label{eq:trap_coords}
  U = Y, \qquad V = S - (n - X),
\end{equation}
At level~$V$, row $S = V + n - X$ carries exactly $V$ arrows
in the active region (one new color-${\le}\,X$ arrow enters from the left at each row).
The \emph{reflecting boundary} lies along the diagonal $S = Y$
of~$\delta$:
\begin{equation}\label{eq:reflecting_bdy}
  U - V = n - X.
\end{equation}

\begin{figure}[t]
\centering
\begin{tikzpicture}[>=stealth, scale=0.55]
  \def\n{10}
  \def\X{4}
  \pgfmathsetmacro{\nX}{\n-\X}          

  \fill[gray!8] (0,0) -- (0,\n) -- (\n,\n) -- cycle;
  \fill[red!12] (0,\nX) -- (0,\n) -- (\n,\n) -- (\nX,\nX) -- cycle;

  \foreach \i in {1,...,9} {
    \draw[gray!25, thin] (0,\i) -- (\i,\i);
    \draw[gray!25, thin] (\i,\i) -- (\i,\n);
  }

  \draw[very thick] (0,0) -- (0,\n) -- (\n,\n) -- cycle;
  \draw[thick, dashed, red!50!black] (0,\nX) -- (\nX,\nX);

  \foreach \i in {1,...,\X} {
    \pgfmathsetmacro{\row}{\nX+\i-0.5}
    \draw[->, red!70!black, thick] (-0.55, \row) -- (0.45, \row);
    \node[left, font=\scriptsize, red!70!black] at (-0.6, \row) {\pgfmathprintnumber[int detect]{\i}};
  }

  \node[font=\large] at (1.5, 2.5) {$\delta$};
  \node[font=\small, red!50!black] at (2.2, 8.5) {trapezoid};

  \node[rotate=45, font=\small, fill=white, inner sep=1pt]
    at ({0.55*\n},{0.55*\n}) {diagonal $S\!=\!Y$};

  \draw[->] (-1.8, 0) -- (-1.8, \n+0.8);
  \node[above] at (-1.8, \n+0.8) {$S$};
  \draw[->] (0, -0.7) -- (\n+0.8, -0.7);
  \node[right] at (\n+0.8, -0.7) {$Y$};

  \draw (-2.0,0) -- (-1.6,0);
  \node[left, font=\small] at (-2.1, 0) {$1$};
  \draw (-2.0,\nX) -- (-1.6,\nX);
  \node[left, font=\small] at (-2.1, \nX) {$n\!-\!X$};
  \draw (-2.0,\n) -- (-1.6,\n);
  \node[left, font=\small] at (-2.1, \n) {$n$};

  \draw (0,-0.9) -- (0,-0.5);
  \node[below, font=\small] at (0, -0.9) {$1$};
  \draw (\n,-0.9) -- (\n,-0.5);
  \node[below, font=\small] at (\n, -0.9) {$n$};

  \draw[->, blue!60!black, thick] (\n+1.3, \nX) -- (\n+1.3, \n+0.8);
  \node[above, blue!60!black] at (\n+1.3, \n+0.8) {$V$};

  \draw[blue!60!black] (\n+1.1,\nX) -- (\n+1.5,\nX);
  \node[right, font=\small, blue!60!black] at (\n+1.6, \nX) {$0$};
  \draw[blue!60!black] (\n+1.1,\n) -- (\n+1.5,\n);
  \node[right, font=\small, blue!60!black] at (\n+1.6, \n) {$X$};

  \node[font=\small, blue!60!black, anchor=north west] at (\n+1.6, \nX-0.5)
    {$U = Y$};
  \node[font=\small, blue!60!black, anchor=north west] at (\n+1.6, \nX-1.4)
    {$V = S-(n\!-\!X)$};
\end{tikzpicture}
\caption{The triangular domain $\delta = \{(S,Y) : 1 \le Y < S \le n\}$
  and the trapezoid subregion (shaded, $S \ge n-X+1$)
  with the coordinate change \eqref{eq:trap_coords}.
  The diagonal $S = Y$ is the reflecting boundary
  of the color-forgotten model.}
\label{fig:triangle}
\end{figure}
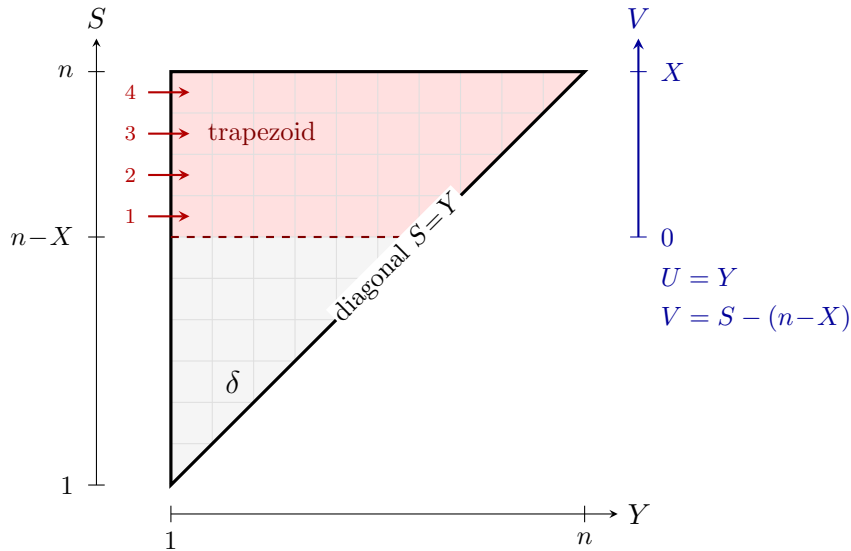

To embed the trapezoid into a cylinder, set $M = n + X$
and identify $\mathfrak{T}_M = \mathbb{Z}/M\mathbb{Z}$.
At each level~$V$, the active region (the trapezoid)
occupies positions $U \in \{1, \dots, V + n - X\}$ on the circle,
To complete the embedding, we fill the complementary arc
with a \emph{frozen block} of $2X - V$ sites, all occupied by arrows.
The reflecting boundary vertex (\cref{sec:model}) acts deterministically, so no arrows cross between the active region and the complement, and this addition does not affect the dynamics inside the trapezoid.
The total number of arrows on each level is therefore $2X$:
$V$ stochastic arrows in the active region and $2X - V$ frozen arrows
in the complement.

We work in the rescaled coordinates $(u, v) = (U/M,\, V/M)$ on the cylinder.
Set
\begin{equation}\label{eq:lambda_def}
  \lambda \;:=\; \frac{x}{1+x},
\end{equation}
so that, since $X = \lfloor xn \rfloor$ and $M = n+X$,
\[
  \frac{X}{M} \;\xrightarrow{\;n\to\infty\;}\; \lambda.
\]
In particular $0 < \lambda \le \frac12$ for $x \in (0,1)$.
The reflecting boundary becomes asymptotically
\begin{equation}\label{eq:reflecting_rescaled}
  u - v \;=\; \frac{n - X}{M} \;\xrightarrow{\;n\to\infty\;}\; 1 - 2\lambda.
\end{equation}
The active region at time $v$ is $\{u \in [0,\, v + 1 - 2\lambda)\}$
and the frozen block is $\{u \in [v + 1 - 2\lambda,\, 1)\}$.

\subsection{Monotone coupling}
\label{sec:coupling}

We compare the reflective system with two S6V systems.

For the lower bound we use the \emph{double-step system} on $\mathfrak{C}_{M,L}$ with double-step initial data
$\Psi(u) = 1 - \mathbf{1}_{0 \le u \le 1-2\lambda}$
from \eqref{eq:initial_conditions_double_step}.
At $V = 0$, this system has the same $2X$ arrows in $[1 - 2\lambda, 1)$
as the reflective system.
However, the diagonal boundary imposes no barrier at the reflecting
boundary position: arrows may freely cross in either direction.

For the \emph{upper bound}, we compare the active particles of the reflective
system directly with the \emph{step-initial S6V on the quadrant}: $V$ particles
at level~$V$, evolving with the same weights $w_{p,q}$ but with no reflecting
boundary constraint.
By \cite{BCG6V}, this system converges to the rarefaction profile $G_{BCG}$
from \eqref{eq:gbcg_def}.

\begin{figure}[t]
\centering
\makebox[\textwidth][c]{%
  \includegraphics[height=4.5cm, keepaspectratio]{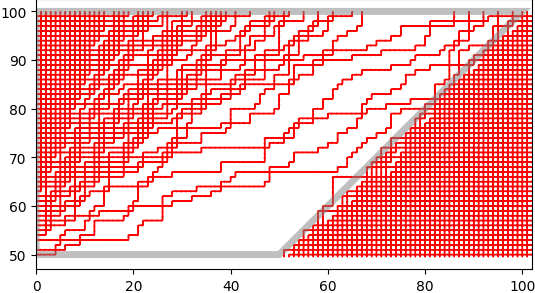}%
  \hspace{0.4cm}%
  \includegraphics[height=4.5cm, keepaspectratio]{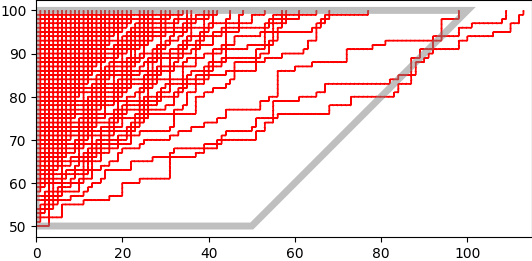}%
}

\vspace{0.15cm}

\includegraphics[width=0.7\textwidth]{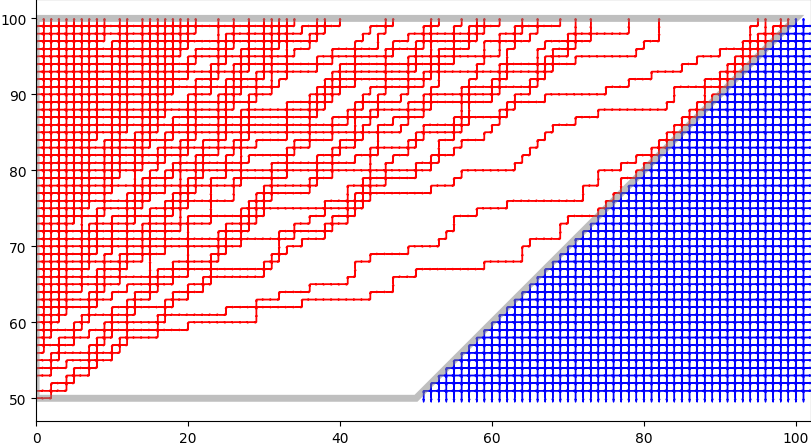}
\caption{The monotone coupling of \cref{lem:monotone_coupling}
  ($n = 100$, $p = 0.7$, $q = 0.5$, $X = 50$).
  Red paths are particle trajectories; axes are spatial coordinate~$U$ (horizontal) and time~$V$ (vertical).
  \emph{Top left:} the double-step system on the cylinder (lower bound on total height);
  frozen-block particles leak freely into the active region.
  \emph{Top right:} the step-initial system on the quadrant (upper bound on active height);
  particles spread past the reflecting boundary (gray line).
  \emph{Bottom:} the reflective system --- the model of interest.
  Active particles (red) are confined to the left of the frozen block (blue).
  The active height function is sandwiched between the two reference systems.}
\label{fig:coupling}
\end{figure}

\medskip

To state the coupling, we cut
$\mathfrak{T}_M$ at position $U = 0$ and work on the interval
$\{0, 1, \dots, M-1\}$.
The reflective height function decomposes as
\begin{equation}\label{eq:height_decomp}
  H^{refl}(V, U) \;=\; H^{act}(V, U) + F_V(U),
\end{equation}
where $H^{act}$ counts active (stochastic) particles and
$F_V$ counts frozen particles at positions $\ge U$.
Since the frozen block is fully packed and no arrows cross the reflecting
boundary, $F_V$ is deterministic.
Away from the reflecting boundary vertex, both the reflective
and double-step systems evolve with
the ordinary S6V weights $w_{p,q}$. At the reflecting boundary vertex, the reflective
system is deterministic: an incoming horizontal arrow from the active region is
forced to turn upward, while in the absence of such an arrow the incoming
vertical arrow from the frozen block is forced to turn right. In particular, no
arrow in the reflective system can cross from the active region into the frozen
block.

The coupled dynamics are arranged so that away from the reflecting boundary,
the reference systems use the same local S6V randomness as the reflective
system. Thus the only place where the reflective system differs from the
double-step and quadrant systems is the reflecting boundary vertex.
Accordingly, the ordering follows from bulk monotonicity
and a direct check at the reflecting boundary vertex.

For any finite particle configuration $\eta$, let
\[
  \mathbf{p}(\eta) = \bigl(p_1(\eta) < \cdots < p_N(\eta)\bigr)
\]
denote the ordered positions of its vertical arrows. If $\eta$ and $\xi$ have
the same number of particles, write
\[
  \mathbf{p}(\eta) \succeq \mathbf{p}(\xi)
\]
if $p_i(\eta) \ge p_i(\xi)$ for every~$i$. Equivalently, this means that the
corresponding suffix-count height functions are ordered pointwise.

\begin{lemma}[{\cite[Proposition~2.6]{aggarwal2020limit}}]
\label{lem:aggarwal_monotone}
Let $\eta_0$ and $\xi_0$ be two finite stochastic six-vertex particle
configurations on $\mathbb{Z}$ with the same number of particles, and assume
\[
  \mathbf{p}(\eta_0) \succeq \mathbf{p}(\xi_0).
\]
Then one can couple one time step of the ordinary stochastic six-vertex
dynamics with weights $w_{p,q}$ so that
\[
  \mathbf{p}(\eta_1) \succeq \mathbf{p}(\xi_1)
\]
almost surely.
\end{lemma}

\begin{lemma}[Monotone coupling]
\label{lem:monotone_coupling}
There exists a coupling such that for all $U \in \{0, 1, \dots, M-1\}$
and $V \in \{0, 1, \dots, X\}$:
\begin{equation}\label{eq:height_ordering}
  H^{refl}(V, U)
  \;\ge\; H^{dstep}(V, U)
  \qquad\text{and}\qquad
  H^{act}(V, U)
  \;\le\; H^{quad}(V, U),
\end{equation}
where $H^{quad}$ is the height function of the step-initial S6V on the quadrant.
\end{lemma}
\begin{proof}
We argue by induction on $V$.
Let $\mathbf r_V$ and $\mathbf d_V$ denote the full particle-position sequences
of the reflective and double-step systems (both have $2X$ particles),
and let $\mathbf r_V^{act}$ and $\mathbf q_V$ denote the active particles of
the reflective system and the quadrant system (both have $V$ particles).

\emph{Base case ($V = 0$).}
The reflective and double-step systems have the same initial configuration,
so $\mathbf r_0 = \mathbf d_0$.
Both active-particle systems are empty.

\emph{Inductive step.}
Assume $\mathbf r_V \succeq \mathbf d_V$ and
$\mathbf q_V \succeq \mathbf r_V^{act}$.
On every vertex strictly to the left of the reflecting boundary, all systems evolve with
the same S6V rule, so \cref{lem:aggarwal_monotone} preserves both orderings
in the bulk. It remains to check the reflecting boundary vertex.

\emph{Lower bound} ($H^{refl} \ge H^{dstep}$):
If a horizontal arrow arrives from the active region, both systems see
the deterministic $(1,1) \mapsto (1,1)$ configuration, so the ordering
is unchanged.
If no horizontal arrow arrives, the reflective system forces the incoming
vertical arrow from the frozen block to turn right (the frozen block stays
packed), while the double-step system applies the ordinary $(1,0)$ stochastic
vertex: the arrow continues upward into the active region with probability~$qp$ and turns right with probability~$1-qp$.
In the latter case one particle shifts weakly to the left,
so $\mathbf r_{V+1} \succeq \mathbf d_{V+1}$.

\emph{Upper bound} ($H^{act} \le H^{quad}$):
If no horizontal active arrow reaches the reflecting boundary, the bulk coupling gives
$\mathbf q_{V+1} \succeq \mathbf r_{V+1}^{act}$.
If a horizontal arrow does reach the reflecting boundary, the reflective system forces it
to turn upward, while the quadrant system applies the ordinary $(0,1)$ stochastic
vertex: the arrow continues to the right with probability~$p$ and turns upward with probability~$1-p$.
In the latter case one particle shifts weakly to the right,
so $\mathbf q_{V+1} \succeq \mathbf r_{V+1}^{act}$.
\end{proof}

\subsection{Limit shape of the reflective system}
\label{sec:permuton}

We identify the limit of $H^{act}$
using \cref{lem:monotone_coupling} and the
hydrodynamic limits from \cref{sec:hydro}.

In rescaled coordinates $(u,v) = (U/M, V/M)$, the reference systems are governed by known hydrodynamic profiles:
\begin{itemize}
  \item The \emph{quadrant system} converges to the rarefaction profile
    $G_{BCG}(u,v)$ from \eqref{eq:gbcg_def}, by \cite{BCG6V}.
  \item The \emph{double-step system} has hydrodynamic profile $G_{shock}(u,v)$ from
    \cref{sec:hydro}, in the sense of \cite[Theorem~1.1]{aggarwal2020limit}.
\end{itemize}
Write $h^{dstep}(u,v) = \int_u^1 G_{shock}(u',v)\,du'$ and $h^{act}(u,v)$ for the hydrodynamic limits of the corresponding discrete height functions, and let $h_{BCG}(u,v)$ be as in \eqref{eq:hbcg_explicit}.
The frozen block contributes a deterministic term $2\lambda - v$ for $u$ in the
active region ($u < v + 1 - 2\lambda$).
By \cref{lem:monotone_coupling}, the active height function is sandwiched:
\begin{equation}\label{eq:sandwich_rescaled}
  h^{dstep}(u,v) - (2\lambda - v)
  \;\le\; h^{act}(u,v)
  \;\le\; h_{BCG}(u,v).
\end{equation}

We now show that both bounds coincide for $u \le \mathcal{V}(v)$.
The upper bound $h^{act} \le h_{BCG}$ holds by construction.
For the lower bound, we split the integral defining $h^{dstep}$
at the shock curve:
\[
  h^{dstep}(u,v)
  = \int_u^{\mathcal{V}(v)} G_{BCG}(u',v)\,du'
    + \int_{\mathcal{V}(v)}^1 1\,du'
  = h_{BCG}(u,v) - h_{BCG}\!\bigl(\mathcal{V}(v),v\bigr)
    + \bigl(1 - \mathcal{V}(v)\bigr).
\]
By \cref{cor:arrow_conservation},
$h_{BCG}\!\bigl(\mathcal{V}(v),v\bigr) = v + 1 - 2\lambda - \mathcal{V}(v)$,
so the last two terms equal $2\lambda - v$, giving
$h^{dstep}(u,v) - (2\lambda - v) = h_{BCG}(u,v)$.
The sandwich \eqref{eq:sandwich_rescaled} collapses on $[0,\, \mathcal{V}(v)]$.

It remains to determine the density to the right of $\mathcal{V}(v)$
when $v > v_{RP}$.
By \cref{cor:arrow_conservation}, $h_{BCG}\!\bigl(\mathcal{V}(v), v\bigr) = v + 1 - 2\lambda - \mathcal{V}(v)$:
the BCG height function at $\mathcal{V}(v)$ accounts for exactly as many arrows
as there is space between the shock curve and the reflecting boundary.
Since $h^{act}(0,v) = v$ (no arrows cross the reflecting boundary) and the sandwich gives $h^{act}(\mathcal{V}(v), v) = h_{BCG}(\mathcal{V}(v), v)$, the remaining mass on
$[\mathcal{V}(v),\, v + 1 - 2\lambda)$ equals the length of the interval.
Since $-\partial_u h^{act} \in [0,1]$, the density must equal~$1$ a.e.\ on this interval.

Combining, the active height function converges to
\begin{equation}\label{eq:h_refl}
  h(u,v) \;=\;
  \begin{cases}
    h_{BCG}(u,v), & u \le \mathcal{V}(v), \\[4pt]
    h_{BCG}\!\bigl(\mathcal{V}(v),\, v\bigr) - \bigl(u - \mathcal{V}(v)\bigr),
      & \mathcal{V}(v) < u < v + 1 - 2\lambda,
  \end{cases}
\end{equation}
for $u$ in the active region $[0,\, v + 1 - 2\lambda)$.
Note that $h(0, v) = v$ and $h(v + 1 - 2\lambda, v) = 0$.

Finally, we rescale back to triangle coordinates.
For a point $(s, y)$ in the rescaled triangle
(where $s = S/n$ and $y = Y/n$),
the corresponding cylinder coordinates are
\begin{equation}\label{eq:rescaling}
  u = \frac{y}{1+x}, \qquad v = \frac{s - (1-x)}{1+x},
\end{equation}
and the rescaling factor is $M/n \to 1 + x$ as $n \to \infty$.
In particular, the top row $s = 1$ corresponds to $v = \lambda = x/(1+x)$,
and $y \in [0, 1]$ maps to $u \in [0,\, 1-\lambda]$.
The uncolored height function rescales as
\begin{equation}\label{eq:height_rescaling}
  \frac{1}{n}\, H_X\!\bigl(\lfloor sn \rfloor,\, \lfloor yn \rfloor\bigr)
  \;\xrightarrow{\;n\to\infty\;}\;
  (1+x)\, h\!\bigl(\tfrac{y}{1+x},\, \tfrac{s-(1-x)}{1+x}\bigr).
\end{equation}

\subsection{Main theorem}
\label{sec:main_thm}

\begin{theorem}[Main theorem]
\label{thm:main}
Fix $q \in (0,1)$ and $p \in (0,1)$.
Let $\sigma_n \in S_n$ be the random permutation obtained by the
$q$-Demazure product with deletion probability $1-p$ applied to the
standard reduced decomposition of~$w_0$.
Then the sequence of permutons $\mu_{\sigma_n}$ converges weakly,
in probability, to a deterministic permuton $\mu_{p,q}$.
If
\begin{equation}
\label{eq:adjusted_p}
  p' \;=\; \frac{p(1-q)}{1-qp}.
\end{equation}
then $\mu_{p,q}$ coincides with the limiting permuton from
\cite{GrothendieckShenanigans2024} for the Demazure product ($q=0$)
with parameter~$p'$.
\end{theorem}
\begin{proof}
Fix $(x,y) \in (0,1)^2$ and set $X = \lfloor xn \rfloor$.
By the color-forgetting property, the uncolored height function
$H_X(n, \lfloor yn \rfloor)$ counts arrows with color $\le X$
at positions $Y' \ge \lfloor yn \rfloor$ on the top row $S = n$.
Since each position $Y'$ carries exactly one color, we have
\[
  \frac{1}{n}\, H_X(n, \lfloor yn \rfloor)
  \;=\; \frac{1}{n}\,\#\bigl\{\, i \ge \lfloor yn \rfloor : \sigma(i) \le X \,\bigr\}
  \;=\; x - y + H_\sigma(y, x) + O(1/n),
\]
where $H_\sigma$ is the permutation height function from \cref{def:intro_h_perm};
the second step uses
$\#\{i \ge Y : \sigma(i) \le X\} = X - (Y-1) + \#\{i \le Y-1 : \sigma(i) > X\}$.
Thus convergence of permutons reduces
to showing that for every $(x, y) \in (0,1)^2$, the rescaled height
function $\frac{1}{n}\, H_{\lfloor xn \rfloor}\!\bigl(n,\, \lfloor yn \rfloor\bigr)$ converges in probability.

Set $\lambda = x/(1+x)$.
By \eqref{eq:height_rescaling} evaluated at $s = 1$
(so $v = \lambda$ and $u = y/(1+x)$), this converges to
$(1+x)\, h\!\bigl(\tfrac{y}{1+x},\, \lambda \bigr)$,
where $h$ is the limiting active height function from \eqref{eq:h_refl}.
Substituting \eqref{eq:h_refl} and \eqref{eq:hbcg_explicit} with $v = \lambda = x/(1+x)$ and $u = y/(1+x)$, we obtain the permutation height function
$H_\sigma(y,x) = (1+x)\, h\!\bigl(\tfrac{y}{1+x},\, \lambda \bigr) - x + y$, where $\kappa = (1-qp)/(1-p)$.

If $x \le \kappa^{-1}$, then $\mathcal{V}(\lambda) = 1-\lambda$, so
$h(\cdot,\lambda) = h_{BCG}(\cdot,\lambda)$ on the whole top row, and
\[
  H_\sigma(y,x) \;=\;
  \begin{cases}
    0, & y \le x/\kappa, \\[6pt]
    \displaystyle \frac{\bigl(\sqrt{\kappa y} - \sqrt{x}\bigr)^2}{\kappa - 1},
      & x/\kappa < y \le x\kappa, \\[8pt]
    y - x, & y > x\kappa.
  \end{cases}
\]

If $x > \kappa^{-1}$, then the shock has detached by time $\lambda$, and
\[
  (1+x)\,\mathcal{V}\!\left(\tfrac{x}{1+x}\right)
  \;=\;
  \frac{\bigl(\sqrt{x} + \sqrt{(1-x)(\kappa-1)}\bigr)^2}{\kappa}.
\]
Using \eqref{eq:h_refl}, \eqref{eq:hbcg_explicit}, and \eqref{eq:arrow_conservation} at $v=\lambda$, we obtain
\[
  H_\sigma(y,x) \;=\;
  \begin{cases}
    0, & y \le x/\kappa, \\[6pt]
    \displaystyle \frac{\bigl(\sqrt{\kappa y} - \sqrt{x}\bigr)^2}{\kappa - 1},
      & x/\kappa < y \le (1+x)\,\mathcal{V}\!\left(\tfrac{x}{1+x}\right), \\[10pt]
    1 - x, & y > (1+x)\,\mathcal{V}\!\left(\tfrac{x}{1+x}\right).
  \end{cases}
\]

The function $h$ depends on the S6V weights only through
$\kappa = (1-qp)/(1-p)$.
Rewriting the limiting height function from
\cite{GrothendieckShenanigans2024} for the Demazure product ($q=0$) with retention probability $p'$,
one sees that it depends on $p'$ only through
$\kappa' = 1/(1-p')$.
Setting $\kappa' = \kappa$ yields
\[
  \frac{1}{1-p'} = \frac{1-qp}{1-p}
  \qquad\Longleftrightarrow\qquad
  p' = \frac{p(1-q)}{1-qp},
\]
confirming \eqref{eq:adjusted_p}.
Since $h$ depends on $p$ and $q$ only through~$\kappa$,
and this~$\kappa$ coincides with the one for the $q=0$ model at parameter~$p'$,
the limiting permuton is the same. We therefore set
\[
  \mu_{p,q} := \mu_{p'}.
\]

Since the above convergence holds for every
$(x,y) \in (0,1)^2$, the permutons $\mu_{\sigma_n}$ converge weakly
to $\mu_{p,q}$ in probability
(since convergence of $H_\sigma$ determines the permuton; see \cite{hoppen2013limits}).
\end{proof}

\bibliography{qdem_bib}

@article{aggarwal2020limit,
	author = {Aggarwal, A.},
	journal = {Commun. Math. Phys.},
	note = {arXiv:1902.10867 [math.PR]},
	number = {1},
	pages = {681--746},
	publisher = {Springer},
	title = {{Limit Shapes and Local Statistics for the Stochastic Six-Vertex Model}},
	volume = {376},
	year = {2020}}

@article{BCG6V,
	author = {Borodin, A. and Corwin, I. and Gorin, V.},
	journal = {Duke Math. J.},
	note = {arXiv:1407.6729 [math.PR]},
	number = {3},
	pages = {563-624},
	title = {{Stochastic six-vertex model}},
	volume = {165},
	year = {2016}}

@article{borodin2019shift,
	author = {Borodin, A. and Gorin, V. and Wheeler, M.},
	doi = {10.1112/plms.12427},
	issue = {2},
	journal = {Proc. Lond. Math. Soc.},
	note = {arXiv:1912.02957 [math.PR]},
	pages = {182-299},
	title = {Shift-invariance for vertex models and polymers},
	url = {https://doi.org/10.1112/plms.12427},
	volume = {124},
	year = {2022}}

@article{BorodinBufetov2021ColorPosition,
	author = {Borodin, A. and Bufetov, A.},
	doi = {10.1214/20-AOP1463},
	journal = {Ann. Probab.},
	note = {arXiv:1905.04692 [math.PR]},
	number = {4},
	pages = {1607--1632},
	title = {{Color-position symmetry in interacting particle systems}},
	volume = {49},
	year = {2021}}

@article{bufetov2020interacting,
	author = {Bufetov, A.},
	journal = {arXiv preprint},
	note = {arXiv:2003.02730 [math.PR]},
	title = {Interacting particle systems and random walks on {H}ecke algebras},
	year = {2020}}

@article{Defant2024,
	author = {Defant, C.},
	journal = {arXiv preprint},
	note = {arXiv:2408.05182 [math.PR]},
	title = {{Random Subwords and Pipe Dreams}},
	year = {2024}}

@article{defant2025permutons,
	author = {Defant, C.},
	journal = {arXiv preprint},
	note = {arXiv:2505.15630 [math.PR]},
	title = {{Permutons from Demazure Products}},
	year = {2025}}

@article{galashin2020symmetries,
	author = {Galashin, P.},
	doi = {10.1214/20-AOP1502},
	journal = {Ann. Probab.},
	note = {arXiv:2003.06330 [math.PR]},
	number = {5},
	pages = {2175-2219},
	title = {Symmetries of stochastic colored vertex models},
	url = {https://doi.org/10.1214/20-AOP1502},
	volume = {49},
	year = {2021}}

@article{GrothendieckShenanigans2024,
	author = {Morales, A.H. and Panova, G. and Petrov, L. and Yeliussizov, D.},
	journal = {arXiv preprint},
	note = {arXiv:2407.21653 [math.PR]. To appear in Adv. Math.},
	title = {{Grothendieck Shenanigans: Permutons from pipe dreams via integrable probability}},
	year = {2024}
}

@article{GwaSpohn1992,
	author = {Gwa, L.-H. and Spohn, H.},
	journal = {Phys. Rev. Lett.},
	number = {6},
	pages = {725--728},
	title = {Six-vertex model, roughened surfaces, and an asymmetric spin {H}amiltonian},
	volume = {68},
	year = {1992}}

@article{hoppen2013limits,
	author = {Hoppen, C. and Kohayakawa, Y. and Moreira, C. G. and R{\'a}th, B. and Sampaio, R. M.},
	doi = {10.1016/j.jctb.2012.09.003},
	journal = {J. Comb. Theory Ser. B},
	note = {arXiv:1103.5844 [math.CO]},
	number = {1},
	pages = {93--113},
	title = {{Limits of permutation sequences}},
	volume = {103},
	year = {2013}}

@article{kruzkov1970first,
  title={{First order quasilinear equations in several independent variables}},
  author={Kru{\v{z}}kov, Stanislav N},
  journal={Mathematics of the USSR-Sbornik},
  volume={10},
  number={2},
  pages={217--243},
  year={1970},
  publisher={IOP Publishing}
}

@article{lascoux1997flag,
	author = {Lascoux, A. and Leclerc, B. and Thibon, J.-Y.},
	journal = {Letters in Mathematical Physics},
	number = {1},
	pages = {75--90},
	publisher = {Springer},
	title = {{Flag varieties and the Yang-Baxter equation}},
	volume = {40},
	year = {1997}}

\end{document}